\documentclass[psamsfonts]{amsart}
  
\usepackage{amssymb,amsfonts}
\usepackage[margin=1.5in]{geometry}
\usepackage{mathtools}
\usepackage{color}
\usepackage[color,matrix,all,arc]{xy}
\usepackage{enumerate}
\usepackage{mathrsfs}
\usepackage[mathscr]{euscript}
\usepackage{graphicx}
\usepackage{float}
\usepackage{url} 
\usepackage{bbm}
\usepackage{soul}
\usepackage{xcolor}
\definecolor{allrefcolors}{rgb}{.05,.45,.6}
\usepackage[pagebackref,linktocpage=true,colorlinks=true,allcolors=allrefcolors,bookmarksopen,bookmarksdepth=3]{hyperref}
\usepackage[noabbrev]{cleveref}
\usepackage{mathtools}
\usepackage{dsfont}
\usepackage{bbm}

\usepackage{tikz}
\usepackage{tqft}
\usepackage{pgfplots}
\usepackage[caption=false]{subfig}

\usepackage{enumitem}
\setenumerate{label=(\roman*),topsep=1pt,itemsep=1pt,partopsep=0pt,parsep=0pt}

\theoremstyle{theorem}
\newtheorem{thm}{Theorem}[section]
\newtheorem{prop}[thm]{Proposition}
\newtheorem{lem}[thm]{Lemma}

\newtheorem{quest}[thm]{Question}

\theoremstyle{definition}
\newtheorem{defn}[thm]{Definition}
\newtheorem{con}[thm]{Construction}

\newtheorem{notn}[thm]{Notation}

\theoremstyle{remark}
\newtheorem{rem}[thm]{Remark}

\crefname{thm}{Theorem}{Theorems}
\crefname{prop}{Proposition}{Propositions}
\crefname{lem}{Lemma}{Lemmas}
\crefname{cor}{Corollary}{Corollaries}

\crefname{defn}{Definition}{Definitions}
\crefname{con}{Construction}{Constructions}
\crefname{exmp}{Example}{Examples}
\crefname{notn}{Notation}{Notations}
\crefname{asmpt}{Assumption}{Assumptions}

\crefname{rem}{Remark}{Remarks}
\crefname{warn}{Warning}{Warnings}

\crefname{section}{Section}{Sections}
\crefname{subsection}{Subsection}{Subsections}

\crefname{sec}{Section}{Sections}
\crefname{subsec}{Subsection}{Subsections}
\crefname{eqn}{Equation}{Equations}
\crefname{part}{Part}{Parts}

\newcommand{\IC}{\mathbb{C}}

\newcommand{\IN}{\mathbb{N}}

\newcommand{\IP}{\mathbb{P}}
\newcommand{\IQ}{\mathbb{Q}}

\newcommand{\IZ}{\mathbb{Z}}

\newcommand{\Ione}{\mathbbm{1}}

\newcommand{\sC}{\mathcal{C}}

\newcommand{\sE}{\mathcal{E}}
\newcommand{\sF}{\mathcal{F}}

\newcommand{\sH}{\mathcal{H}}

\newcommand{\sM}{\mathcal{M}}

\newcommand{\sP}{\mathcal{P}}

\newcommand{\fo}{\mathfrak{o}}

\newcommand{\gap}{\hspace{15pt}}
\newcommand{\wt}[1]{\widetilde{#1}}
\newcommand{\pgl}{\PGL_2(\IC)}

\DeclareMathOperator{\Ima}{Im}

\DeclareMathOperator{\Diff}{Diff}
\DeclareMathOperator{\pt}{pt}
\DeclareMathOperator{\Bd}{Bd}

\DeclareMathOperator{\PGL}{PGL}
\DeclareMathOperator{\GW}{GW}
\DeclareMathOperator{\fat}{fat}
\DeclareMathOperator{\ev}{ev}

\begin{document}

\title{Fundamental groups of rationally connected symplectic manifolds}

\author{Alex Pieloch}

\address{Department of Mathematics\\ Massachusetts Institute of Technology \\ 77 Massachusetts Avenue
\\ Cambridge, MA 02139 }
\email{pieloch@mit.edu}

\thanks{
The author was supported by a National Science Foundation Postdoctoral Research Fellowship through NSF grant DMS-2202941 as well as the Simons Foundation through its ``Homological Mirror Symmetry'' Collaboration grant.}

\begin{abstract}
We show that the fundamental group of every enumeratively rationally connected closed symplectic manifold is finite.  In other words, if a closed symplectic manifold has a non-zero Gromov-Witten invariant with two point insertions, then it has finite fundamental group.  We also show that if the spherical homology class associated to such a non-zero Gromov-Witten invariant is holomorphically indecomposable, then the rational second homology of the symplectic manifold has rank one.
\end{abstract}

\maketitle
\tableofcontents

\section{Introduction}\label{sec:Introduction}

\subsection{Statement of results}\label{subsec:StatementsOfResults}

The main result of this paper concerns the following two classes of symplectic manifolds.

\begin{defn}
A symplectic manifold $(X,\Omega)$ is \emph{enumeratively uniruled} if there exists a spherical second homology class $A \in H_2(X;\IZ)$ and non-torsion homology classes $A_1,\dots,A_k \in H_\bullet(X;\IZ)$ such that
\[ \GW_{0,1+k}^A(X,\Omega;[\pt],A_1,\dots,A_k) \neq 0. \]
A symplectic manifold $(X,\Omega)$ is \emph{enumeratively rationally connected} if there exists a spherical second homology class $A \in H_2(X;\IZ)$ and non-torsion homology classes $A_1,\dots,A_k \in H_\bullet(X;\IZ)$ such that
\[ \GW_{0,k+2}^A(X,\Omega;[\pt],[\pt],A_1,\dots,A_k) \neq 0. \]
\end{defn}

The goal of this paper is to establish the following result.

\begin{thm}\label{thm:MainTheorem}
If $(X,\Omega)$ enumeratively uniruled with
\[ \GW_{0,2+k}^A(X,\Omega;[\pt],[Z],A_1,\dots,A_k) \neq 0, \]
where $h: Z \to X$ is an embedded submanifold of $X$, then the image of
\[ h_*: \pi_1(Z) \to \pi_1(X) \]
is a finite index subgroup of $\pi_1(X)$.  Moreover, if $A \in H_2(X;\IZ)$ is $\Omega$-indecomposable, then $H_2(X;\IQ)$ is generated by $h_*\left(H_2(Z;\IQ)\right)$ and $A$.
\end{thm}

We have the following immediate corollary.

\begin{thm}\label{thm:MainCor}
If $(X,\Omega)$ is enumeratively rationally connected, then $\pi_1(X)$ is finite.  Moreover, if $(X,\Omega)$ is enumeratively rationally connected with 
\[ \GW_{0,2+k}^A(X,\Omega;[\pt],[\pt],A_1,\dots,A_k) \neq 0 \]
for some $\Omega$-indecomposable homology class $A \in H_2(X;\IZ)$, then $H_2(X;\IQ)$ is generated by the class $A$.
\end{thm}

The motivation for our main result comes from algebraic geometry.  A smooth complex projective variety is \emph{rationally connected} if there exists a (possibly nodal) genus zero curve passing through any two generic points in the variety.  A classical result \cite{Campana_OnTwistorSpacesOfTheClassC} (cf. \cite[Corollary 4.18]{Debarre_HigherDimensionalAlgebraicGeometry}) says that every rationally connected variety has trivial fundamental group.  An argument using Hilbert schemes is used to show that such varieties have finite fundamental groups and Hodge theory is used to show that rationally connected varieties with finite fundamental groups have, in fact, trivial fundamental groups.  This leads to an obvious question in the symplectic case.

\begin{quest}
Does there exist an enumeratively rationally connected symplectic manifold with non-trivial fundamental group?
\end{quest}

\begin{rem}
In the specific cases where rational connectedness is known to be equivalent to enumerative rational connectedness, \cref{thm:MainCor} provides another proof of the fact that rationally connected varieties have finite fundamental groups.  In full generality, it is an open question whether or not every rationally connected variety is enumeratively rationally connected.  The only known cases are in complex dimensions one, two, and three.  See \cite{Tian_RationallyConnectedThreefolds} for further discussion and proofs of the known cases.
\end{rem}

\begin{rem}
By Koll\'{a}r-Miyaoka-Mori \cite[Theorem 0.1]{KollarMiyaokaMori_RationalConnectednessAndBoundednessOfFanoManifolds}, a smooth complex projective Fano variety is rationally connected and, thus, is simply connected.  Given \cref{thm:MainCor}, one could ask if every symplectic Fano manifold\footnote{A symplectic manifold $(X,\Omega)$ is \emph{Fano} if $\Omega = \kappa \cdot c_1(TX)$ for some positive constant $\kappa>0$.} has finite fundamental group or is enumeratively rationally connected.  However, this fails to be the case.  Indeed, Panov-Fine \cite{FinePanov_HyperbolicFano} construct non-K\"{a}hler symplectic Fano manifolds that have infinite fundamental groups.
\end{rem}

\begin{rem}
The literature contains several approaches to defining Gromov-Witten invariants, see \cite{FukayaOno_ArnoldConjectureAndGromovWittenInvariant,LiTian_VirtualModuliCyclesAndGromovWittenInvariantsOfGeneralSymplecticManifolds,Ruan_VirtualNeighborhoodsAndPseudoHolomorphicCurves,Siebert_GromovWittenInvariantsOfGeneralSymplecticManifolds,HoferWysockiZehnder_ApplicationsOfPolyfoldTheory1,Pardon_AlgebraicApproachToVFC,AbouzaidMcLeanSmith_HamiltonianLoops,BaiXu_IntegerTypeGWInvariants}, which may or may not all agree.  In this paper, we have elected to use Cieliebak-Mohnke's approach to defining Gromov-Witten invariants via rational pseudocycles, see \cite{CieliebakMohnke_HypersurfacesAndTransversality}.  This was done to minimize the amount of necessary technical machinery.  However, we believe that the results of this paper will continue to hold for all of the different constructions of Gromov-Witten invariants above.  Elements of our argument will simply need to be appropriately modified to suit the flavor of Gromov-Witten invariants and virtual fundamental cycles used.
\end{rem}

\subsection{Idea of proof}\label{subsec:IdeaOfProof}

To illustrate the idea behind \cref{thm:MainTheorem}, consider the case where
\[ \GW_{0,2}^A(X,\Omega;[\pt],[\pt]) = q \not = 0. \]
Let us also assume that we are in a nice situation where the moduli space of stable genus zero $J$-holomorphic maps with $2$-marked points of class $A \in H_2(X;\IZ)$, denoted $\overline{\sM}_{0,2}^A(X,J)$, is a smooth closed oriented manifold of the expected dimension (which in this case is $4n$), where $J$ is some regular $\Omega$-tamed almost complex structure.

We have an evaluation map
\[ \ev = (\ev_0,\ev_1): \overline{\sM}_{0,2}^A(X,J) \to X \times X. \]
For generic $p_1 \in X$, $M \coloneqq \ev^{-1}(X \times \{p_1\})$ is a smooth closed oriented manifold of dimension $2n$ such that:
\begin{enumerate}
	\item $\ev_0|_{M}$ is a degree $q$ map to $X$.  This follows from the non-vanishing Gromov-Witten invariant assumption and Poincar\'{e} duality.
	\item $\ev_1|_{M} \equiv p_1$.  This follows by definition of $M$.
\end{enumerate}
By covering space theory, see \cref{lem:PseudoFiniteIndex}, the image of
\[ (\ev_0|_{M})_* : \pi_1(M) \to \pi_1(X) \]
is a finite index subgroup of $\pi_1(X)$ (this fact continues to hold even when $\ev_0|_M$ is a non-zero pseudocycle of dimension $2n$).  The maps $\ev_0|_{M}$ and $\ev_1|_{M}$ are not homotopic; however, they are homotopic over any generic $1$-skeleton of $M$.  To see this, recall that the moduli space $\overline{\sM}_{0,2}^A(X,J)$ carries a universal curve
\[ \pi: \sC_{0,2}^A(X,J) \to \overline{\sM}_{0,2}^A(X,J), \]
where $\pi^{-1}([u])$ is given by the domain of the stable map $u$.  $\pi$ has (tautological) sections $s_0$ and $s_1$ and an evaluation map
\[ e: \sC_{0,2}^A(X,J) \to X \]
that satisfy $\ev_i = e \circ s_i$.  Generically, the locus of maps in $\overline{\sM}_{0,2}^A(X,J)$ with nodal domains has codimension two.  So for a generic choice of $1$-skeleton in $M$, say $M_{(1)}$, the restriction $\pi|_{M_{(1)}}$ defines an oriented $S^2$-bundle.  The obstruction for the two sections $s_i|_{M_{(1)}}$ to be homotopic lies in the fundamental group of the fibre of $\pi|_{M_{(1)}}$, which is trivial.  Since $\ev_i = e \circ s_i$, the $\ev_i|_{M}$ induce isomorphic maps on $\pi_1(M)$.  But $(\ev_0|_{M})_*$ has finite index and it agrees with the trivial map (it's isomorphic to $(\ev_1|_M)_*$), so it follows that $\pi_1(X)$ is finite.

The restriction on $H_2(X;\IQ)$ is obtained as follows.  When $M$ (as defined above) is a closed manifold (which in general can not be arranged because $\overline{\sM}_{0,2}^A(X,J)$ is hardly every a closed manifold; however, it can be arranged in the case where $A$ is $\Omega$-indecomposable), Poincar\'{e} duality, see \cref{lem:PositiveDegreeAndSurj}, gives a surjection 
\[ (\ev_0|_{M})_*: H_2(M;\IQ) \to H_2(X;\IQ). \]

In the case where $A$ is $\Omega$-indecomposable, one can also arrange for the restriction of the universal curve to $M$, $\pi|_M$, (or an appropriate replacement, see \cref{subsec:Indecomposable}) to be an oriented $S^2$-bundle over all of $M$.  So the Gysin sequence, see \cref{lem:SphereBundleProperties}, implies that for each $\Sigma \in H_2(M;\IQ)$, there exists $c \in \IQ$ such that
\[ c \cdot e_*([\mbox{Fibre}]) = (\ev_0)_*(\Sigma). \]
Here we use that $\ev_i = e \circ s_i$ and that $\ev_1|_M$ is trivial on second homology.
In essence, the difference of the push-forward of any second homology class in $H_2(M;\IQ)$ under two different sections of $\pi|_M$ is a multiple of the fibre.  The class of $e_*([\mbox{Fibre}])$ is given by the class of a map in $\overline{\sM}_{0,2}^A(X,J)$, that is, $A$.  Since $(\ev_0)_*$ is surjective on homology, it follows that every class in $H_2(X;\IQ)$ can be written as a multiple of $A$.

The issue with the above sketch of proof is that $\overline{\sM}_{0,2}^A(X,J)$ is hardly every a smooth closed oriented manifold.  In general, one has to represent it by a Kuranishi structure or represent its virtual homology class by a pseudocycle.  We have opted to use Cieliebak-Mohnke's \cite{CieliebakMohnke_HypersurfacesAndTransversality} approach to define Gromov-Witten invariants so that we can effectively replace $\overline{\sM}_{0,2}^A(X,J)$ by a pseudocycle in the above sketch.  The technical aspects of this paper revolve around altering the above arguments to the cases where the moduli spaces in question are pseudocycles.

\subsection{Outline of paper}\label{subsec:OutlineOfPaper}

In \cref{sec:ProofOfMainTheorem}, we give the proof of our main result, assuming definitions and material from later sections.  In \cref{sec:GWInvariants}, we review and study Cieliebak-Mohnke's \cite{CieliebakMohnke_HypersurfacesAndTransversality} construction of Gromov-Witten moduli spaces and Gromov-Witten invariants.  We also construct a naive, smooth compactification of the open stratum of the moduli space of stable maps that represent an indecomposable spherical class.  In \cref{sec:Pseudocycles}, we derive an intersection product for pseudocycle homology, show that it is well-defined, and establish its additional properties that we need to prove our main result.  In \cref{sec:EasyAT}, we discuss some algebraic topology lemmas that we need to prove our main result.

\subsection*{Acknowledgements}
I thank Dusa McDuff for posing the first conclusion of \cref{thm:MainCor} as a question to me.  I also thank Mohammed Abouzaid and Paul Seidel for discussions related to this paper.


\section{Proof of main result}\label{sec:ProofOfMainTheorem}

In this section, we prove \cref{thm:MainTheorem}.  First, we prove a general geometric result, \cref{prop:GeneralSetup}, that encapsulates the hypotheses of \cref{thm:MainTheorem}.  Then we show that \cref{thm:MainTheorem} follows from \cref{prop:GeneralSetup}.

\begin{notn}\label{notn:GeneralNotn}
Let $X$ be a smooth closed oriented manifold of dimension $n$.  Let $f = (f_0,f_1): \sM \to X \times X$ be a $k$-pseudocycle.  Let $\pi: \sC \to \sM$ be an oriented $S^2$-bundle.  Suppose there is a map $e: \sC \to X$ and sections $s_i: \sM \to \sC$ of $\pi$ that satisfy $f_i = e \circ s_i$.
\end{notn}

\begin{prop}\label{prop:GeneralSetup}
If $h: Z \to X$ is a $(2n-k)$-dimensional embedded submanifold such that
\[ [ f ] \cdot [ \pt \times h ] \neq 0,\footnote{Here we are denoting the intersection product of pseudocycles, see \cref{lem:IntersectingPseudocycles} and \cref{prop:IntersectionProduct}.} \]
then the image of
\[ h_*: \pi_1(Z) \to \pi_1(X) \]
is a finite index subgroup of $\pi_1(X)$.  If $\sM$ is further a closed manifold, then $H_2(X;\IQ)$ is generated by $h_*(H_2(Z;\IQ))$ and $e_*([\pi^{-1}(x)])$.
\end{prop}

\begin{proof}
We begin by proving the first claim about fundamental groups.  By \cref{lem:FactorTransverseImpliesIoneTransverse}, \cref{lem:FactorTransversality}, and \cref{lem:IntersectingPseudocycles}, there exists a diffeomorphism $\phi \in \Diff_0^+(X)$ such that
\[ g: M \to X, \gap g = pr_0 \circ (f \cdot (\Ione \times \phi \circ h))\]
is an $n$-pseudocycle.  We claim that the image of
\[ g_*: \pi_1(M) \to \pi_1(X) \]
is a finite index subgroup of $\pi_1(X)$.  Indeed, by \cref{prop:KunnethForPseudocycles},
\[ [g] \cdot [pt] = [f] \cdot [ pt \times h ] \neq 0. \]
So by \cref{prop:IntersectionProduct}, $[g] \in \sH_n(X)$ is non-zero.  So the claim follows from \cref{lem:PseudoFiniteIndex}.

Since $\Ione \times (\phi \circ h)$ is injective, by \cref{lem:IntersectingPseudocycles}, $M$ naturally sits inside $\sM$.  So we may restrict the $S^2$-bundle over $\sM$ to $M$, $\pi|_M: \sC|_M \to M$.  The homotopy long exact sequence for a fibre bundle implies
\begin{enumerate}
	\item $\pi_*: \pi_1(\sC|_M) \to \pi_1(M)$ is an isomorphism and
	\item the $(s_i|_M)_*: \pi_1(M) \to \pi_1(\sC|_M)$ are isomorphic for $i = 0,1$.
\end{enumerate}
Consider
\[ \widetilde{g}: M \to X, \gap \wt{g} = pr_1 \circ (f \cdot (\Ione \times \phi \circ h)).\]
Since $e \circ s_i = f_i$, the above discussion implies that
\[ g_*: \pi_1(M) \to \pi_1(X) \gap \mbox{and} \gap \wt{g}_*: \pi_1(M) \to \pi_1(X) \]
are isomorphic.  In particular, $\Ima(\wt{g}_*) \subset \pi_1(X)$ is a finite index subgroup.  But
\[ \Ima(\wt{g}_*) \subset \Ima((\phi \circ h)_*) = \Ima(h_*) \subset \pi_1(X). \]
So $\Ima(h_*)$ is a finite index subgroup of $\pi_1(X)$, as desired.

Now we prove the second conclusion.  Assume that $\sM$ is closed.  In this case, $M$ is a closed manifold, and, like above, the class of $g$ in $H_n(X;\IZ)$ is non-zero.  So by \cref{lem:PositiveDegreeAndSurj},
\[ g_*: H_2(M;\IQ) \to H_2(X;\IQ) \]
is surjective.  As above, $\pi|_{M}$ is an oriented $S^2$-bundle.  \cref{lem:SphereBundleProperties} gives an exact sequence
\[ \xymatrix{ H_2(S^2;\IQ) \ar[r]^{i_x} & H_2(\sC|_M;\IQ) \ar[r]^\pi & H_2(M;\IQ) \ar[r] & 0. } \]
Fix some homology class $B \in H_2(X;\IQ)$.  There exists $\wt{B} \in H_2(M;\IQ)$ so that $g_*(\wt{B}) = B$.  By the exact sequence,
\[ q \cdot (i_x)_*([S^2]) = -(s_1)_*(\wt{B}) + (s_0)_*(\wt{B}) \]
for some $q \in \IQ$.  Since $e \circ s_i = f_i$, it follows that
\[ g_*(\wt{B}) = e_* \circ (s_0)_*(\wt{B}) = e_*(q \cdot (i_x)_*([S^2])) + e_* \circ (s_1)_*(\wt{B}).\]
But $e_* \circ (s_1)_*(\wt{B})$ lies in the image of $h_*$.  So the desired result follows.
\end{proof}

We now give the proof of \cref{thm:MainTheorem}.

\begin{proof}[Proof of \cref{thm:MainTheorem}]
We begin with the claim about fundamental groups.
Consider a choice of Gromov-Witten pseudocycle from \cref{thm:CMTheorem},
\[ \ev_{\ell+(2+k)}(X,\Omega;\omega,\mathbf{J},Y): \sM_{\ell+(2+k)}^A(X,\Omega;\omega,\mathbf{J},Y) \to X^{2+k}.\]
By assumption and \cref{defn:GWInvariants}, there are pseudocycles $h_i: Z_i \to X$ such that
\[ [ \ev_{\ell+(2+k)}(X,\Omega;\omega,\mathbf{J},Y)] \cdot  [ \pt \times h \times h_3 \times \cdots \times h_{2+k} ] \neq 0. \]
By \cref{lem:FactorTransverseImpliesIoneTransverse} and \cref{lem:FactorTransversality}, there exists $\phi \in \Diff_0^+(X^{k})$ such that the pseudocycle $\ev_{\ell+(2+k)}(X,\Omega;\omega,\mathbf{J},Y)$ is strongly transverse to
\[ \Ione \times \Ione \times (\phi \circ ( h_3 \times \cdots \times h_{2+k})). \]
Using \cref{lem:IntersectingPseudocycles}, define $f: \sM \to X \times X$ by
\[ f = (pr_1 \times pr_2) \circ ( \ev_{\ell+(k+2)}(X,\Omega;\omega,\mathbf{J},Y) \cdot (\Ione \times \Ione \times (\phi \circ ( h_3 \times \cdots \times h_{2+k})))). \]
By \cref{prop:KunnethForPseudocycles},
\[ [f] \cdot [ pt \times h ] \neq 0. \]
Consider the oriented $S^2$-bundle from \cref{lem:CMUniversalCurve},
\[ \pi: \sC_{0,\ell+(2+k)}^A(X,\Omega;\omega,\mathbf{J},Y) \to \sM_{0,\ell+(2+k)}^A(X,\Omega;\omega,\mathbf{J},Y)\]
with its sections $s_i$ and evaluation map $e$.  Consider the oriented $S^2$-bundle
\begin{align*}
\pi \times \Ione: & \ \sC_{0,\ell+(2+k)}^A(X,\Omega;\omega,\mathbf{J},Y) \times (X \times X \times Z_3\times \cdots Z_{k+2}) 
\\ & \ \ \ \ \ \to \sM_{0,\ell+(2+k)}^A(X,\Omega;\omega,\mathbf{J},Y) \times (X \times X \times Z_3\times \cdots Z_{k+2}).
\end{align*}
It has sections $s_i \times \Ione$ and an evaluation map
\[ e \circ pr: \sC_{0,\ell+(2+k)}^A(X,\Omega;\omega,\mathbf{J},Y) \times (X \times X \times Z_1\times \cdots Z_{k+2}) \to X,\]
where $pr$ denotes the projection onto $\sC_{0,\ell+(2+k)}^A(X,\Omega;\omega,\mathbf{J},Y)$.  Since $\sM$ naturally sits inside the target of $\pi \times \Ione$ above, we may restrict this bundle, its sections, and its evaluation to $\sM$ to obtain the data described in \cref{notn:GeneralNotn}.  The first hypotheses of \cref{prop:GeneralSetup} are now satisfied.  So the first part of \cref{thm:MainTheorem} follows from the first part of \cref{prop:GeneralSetup}.

Now we prove the second conclusion of \cref{thm:MainTheorem}.  Consider a choice of Gromov-Witten pseudocycle from \cref{prop:NaiveCompactificationOfIndecomp},
\[ \widehat{\ev}: \widehat{\sM}_{0,2+k}^A(X,J) \to X^{2+k}.\]
Recall that every rational homology class can be represented as some multiple of an embedded submanifold.  So by assumption, \cref{defn:GWInvariants}, and \cref{rem:GWInvariantsAgree}, there exist submanifolds $h_i: Z_i \to X$ such that
\[ [ \widehat{\ev}] \cdot  [ \pt \times h \times h_3 \times \cdots \times h_{2+k} ] \neq 0. \]
By \cref{lem:FactorTransversality}, there exists $\phi \in \Diff_0^+(X^{k})$ such that $\widehat{\ev}$ is strongly transverse to
\[ \Ione \times \Ione \times (\phi \circ ( h_3 \times \cdots \times h_{2+k})). \]
Define
\[ f: \sM \to X \times X, \gap f = (pr_1 \times pr_2) \circ \widehat{\ev} \cdot (\Ione \times \Ione \times (\phi \circ ( h_3 \times \cdots \times h_{2+k}))). \]
In this case, $\sM$ is a closed manifold and, by \cref{prop:KunnethForPseudocycles},
\[ [f] \cdot [pt \times h] \neq 0. \]
Consider the oriented $S^2$-bundle from \cref{prop:NaiveCompactificationOfIndecomp}:
\[ \pi: \widehat{\sM}_{0,k+1}^A(X,J) \to  \widehat{\sM}_{0,k}^A(X,J). \]
Using \cref{prop:NaiveCompactificationOfIndecomp} and defining $\sC \coloneqq \pi^{-1}(\sM)$, $s_i = s_i|_{\sM}$ for $i=0,1$, and $e = pr_{k+3} \circ \widehat{\ev}$ gives the desired data from \cref{notn:GeneralNotn}.  The second hypotheses of \cref{prop:GeneralSetup} are satisfied.  So the rest of \cref{thm:MainTheorem} follows from the second part of \cref{prop:GeneralSetup}.
\end{proof}


\section{Gromov-Witten invariants}\label{sec:GWInvariants}

\subsection{Gromov-Witten invariants \'{a} la Cieliebak-Mohnke}\label{subsec:CMApproach}

We recall Cieliebak-Mohnke's \cite{CieliebakMohnke_HypersurfacesAndTransversality} construction of Gromov-Witten pseudocycles for rational symplectic manifolds and Krestiachine's \cite{Krestiachine_DonaldsonGW} extension of their construction to non-rational symplectic manifolds.

\begin{notn}
\begin{enumerate}
	\item Let $(X,\Omega)$ be a closed symplectic manifold.
	\item Let $\omega$ be a perturbation of $\Omega$ to a rational symplectic form.  So $[\omega] \in H^2(X;\IQ)$.  After rescaling, we can and do assume that $[\omega] \in H^2(X;\IZ)$.
	\item Let $J$ be an $\omega$-compatible almost complex structure.
	\item Let $Y$ denote a Donaldson hypersurface for $\omega$, that is, $PD([Y]) = N \cdot [\omega]$ for some $N\gg 0$.  Further suppose that $(Y,J)$ is a \emph{Donaldson pair}, see \cite[Definition 9.2]{CieliebakMohnke_HypersurfacesAndTransversality}.  Such pairs always exist for any fixed compatible pair $(\omega,J)$.
\end{enumerate}
\end{notn}

\begin{defn}
A  family of $\omega$-tamed almost complex structures
\[ \mathbf{J} \in C^\infty(\overline{\sM}_{0,\ell+1}, End(TM))\]
is \emph{admissible} if $\mathbf{J}$ is \emph{coherent} (see \cite[Definition 3.1]{CieliebakMohnke_HypersurfacesAndTransversality}) and for each $p \in \overline{\sM}_{0,\ell+1}$, $\mathbf{J}(p)$ leaves $TY$ invariant and $\mathbf{J}(p)$ is $C^0$-close to $J$.  Here $\overline{\sM}_{0,\ell+1}$ is the moduli space of stable genus zero curves with $\ell+1\geq 4$ ordered marked points.
\end{defn}

\begin{notn}
By forgetting the $(\ell+1)$-marked point and restricting to the open stratum $\sM_{0,\ell+1} \subset \overline{\sM}_{0,\ell+1}$, an admissible $\mathbf{J}$ encodes a ${\sM}_{0,\ell}$-family of $\IP^1$ domain dependent almost complex structures on $X$, where the $\IP^1$ dependence is parameterized by the $(\ell+1)$-marked point.  A tuple $\mathbf{z} = (z_1,\dots,z_\ell,\dots,z_{\ell+k}) \in (\IP^1)^{\ell+k} \smallsetminus \Delta_{\fat}$ determines an element $[\mathbf{z}] \in \sM_{0,\ell}$ by forgetting the last $k$-marked points.  Let $\mathbf{J}_\mathbf{z}$ denote the associated $\IP^1$-dependent almost complex structure on $X$.
\end{notn}

\begin{notn}
Let $\wt{\sM}_{0,\ell+k}^A(X,\Omega;\omega,\mathbf{J},Y)$ denote the subspace of
\[ C^\infty(\IP^1,X) \times \left( (\IP^1)^{\ell+k} \smallsetminus \Delta_{\fat} \right) \]
of pairs $(u,\mathbf{z})$ that satisfy:
\begin{enumerate}
	\item $du(z) + \mathbf{J}_{\mathbf{z}}(z,u(z)) \circ du(z) \circ j(z)$,
	\item $[u] = A \in H_2(X;\IZ)$, and
	\item $u(z_i) \in Y$ for $1 \leq i \leq \ell$.
\end{enumerate}
Define an evaluation map
\[ \wt{\ev}_{\ell+k}(X,\Omega;\omega,\mathbf{J},Y): \wt{\sM}_{0,\ell+k}^A(X,\Omega;\omega,\mathbf{J},Y) \to X^k \]
by
\[ \wt{\ev}_{\ell+k}(X,\Omega;\omega,\mathbf{J},Y)(u,z_1,\dots,z_{\ell+k}) = (u(z_{\ell+1}),\dots,u(z_{\ell+k})). \]
By \cite[Section 5]{CieliebakMohnke_HypersurfacesAndTransversality}, there is a proper free $\pgl$-action on $\wt{\sM}_{0,\ell}^A(X,\Omega;\omega,\mathbf{J},Y)$,
\[ \psi \cdot (u,z_1,\dots,z_{\ell+k}) \mapsto (u \circ \psi^{-1}, \psi(z_1),\dots,\psi(z_{\ell+k})). \]
Let $\sM_{0,\ell+k}^A(X,\Omega;\omega,\mathbf{J},Y)$ denote the quotient of $\wt{\sM}_{0,\ell+k}^A(X,\Omega;\omega,\mathbf{J},Y)$ by the $\pgl$-action and let ${\ev}_{\ell+k}(X,\Omega;\omega,\mathbf{J},Y)$ denote the induced evaluation map.
\end{notn}

The main result of \cite{CieliebakMohnke_HypersurfacesAndTransversality} is the following:

\begin{thm}\label{thm:CMTheorem}\cite[Theorem 1.1 and Theorem 1.2]{CieliebakMohnke_HypersurfacesAndTransversality}
For a generic choice of admissible $\mathbf{J}$ and $\ell \gg 0$, the evaluation map
\[ \ev_{\ell+k}(X,\Omega;\omega,\mathbf{J},Y): \sM_{0,\ell+k}^A(X,\Omega;\omega,\mathbf{J},Y) \to X^k \]
defines a pseudocycle of dimension $2n+2c_1(A) +2k-6$.
\end{thm}

\cref{thm:CMTheorem} allows us to give the following definition of Gromov-Witten invariants.

\begin{defn}\label{defn:GWInvariants}
Let $A \in H_2(X;\IZ)$ be a spherical homology class.  Let $A_1,\dots,A_k \in H_\bullet(X;\IZ)$ be non-torsion homology classes that satisfy:
\[ \sum_{i=1}^k 2n-\deg(A_i) = 2n+ 2c_1(A)+2k-6. \]
Suppose that $A_i$ is represented by a pseudocycle $h_i: Z_i \to X$.
Define
\[ \GW_{0,k}^A(X,\Omega;A_1,\dots,A_k) \coloneqq  \frac{1}{\ell!} \cdot \Phi([\ev_{\ell+k}(X,\Omega;\omega,\mathbf{J},Y)] \cdot [h_1 \times \cdots \times h_k]), \]
where $\ev_{\ell+k}(X,\Omega;\omega,\mathbf{J},Y)$ is a pseudocycle from \cref{thm:CMTheorem} and $\Phi$ is given in \cref{thm:ZingerIsomorphism}.
\end{defn}

By \cite[Theorem 1.2]{CieliebakMohnke_HypersurfacesAndTransversality} and \cite[Theorem B]{Krestiachine_DonaldsonGW}, \cref{defn:GWInvariants} is well-defined.  Finally, we give a type of universal curve over these pseudocycles.

\begin{lem}\label{lem:CMUniversalCurve}
There exists an oriented $S^2$-bundle
\[ \pi: \sC_{0,\ell+k}^A(X,\Omega;\omega,\mathbf{J},Y) \to \sM_{0,\ell+k}^A(X,\Omega;\omega,\mathbf{J},Y) \]
with sections
\[ s_i: \sM_{0,\ell+k}^A(X,\Omega;\omega,\mathbf{J},Y) \to \sC_{0,\ell+k}^A(X,\Omega;\omega,\mathbf{J},Y) \]
and an evaluation map
\[ e: \sC_{0,\ell+k}^A(X,\Omega;\omega,\mathbf{J},Y) \to X\]
that satisfy
\[ e \circ s_i = pr_i \circ \ev_{\ell+k}(X,\Omega;\omega,\mathbf{J},Y).\]
\end{lem}

\begin{proof}
Consider the projection $\wt{\pi}$ with sections $\wt{s}_i$:
\begin{align*}
\wt{\pi}: \wt{\sM}_{\ell+k}^A(X,\Omega;\omega,\mathbf{J},Y) \times \IP^1 \to \wt{\sM}_{\ell+k}^A(X,\Omega;\omega,\mathbf{J},Y), \\
\wt{s}_i: \wt{\sM}_{\ell+k}^A(X,\Omega;\omega,\mathbf{J},Y) \to \wt{\sM}_{\ell+k}^A(X,\Omega;\omega,\mathbf{J},Y) \times \IP^1,
\\ \wt{s}_i(u,z_1,\dots,z_{\ell+k}) = (u,z_1,\dots,z_{\ell+k},z_{\ell+i}).
\end{align*}
There is an evaluation map 
\[ \wt{e}: \wt{\sM}_{\ell+k}^A(X,\Omega;\omega,\mathbf{J},Y) \times \IP^1 \to X, \gap \wt{e}: (u,z_1,\dots,z_{\ell+k},z) = u(z). \]
By \cite[Section 5]{CieliebakMohnke_HypersurfacesAndTransversality}, the action of $\pgl$ on $\wt{\sM}_{\ell+k}^A(X,\Omega;\omega,\mathbf{J},Y)$ is proper and free.  Since $\pgl$ is connected, its action on both $\wt{\sM}_{\ell+k}^A(X,\Omega;\omega,\mathbf{J},Y)$ and $\IP^1$ is orientation preserving.  So by \cref{lem:QuotientS2BundleOrientation}, the quotient
\[\sC_{0,\ell+k}^A(X,\Omega;\omega,\mathbf{J},Y) \coloneqq \left( \wt{\sM}_{\ell+k}^A(X,\Omega;\omega,\mathbf{J},Y) \times \IP^1\right)/\pgl \]
by the diagonal action is an oriented $S^2$-bundle over $\sM_{0,\ell+k}^A(X,\Omega;\omega,\mathbf{J},Y)$.   Finally, the maps $\wt{s_i}$ and $\wt{e}$ are equivariant with respect to the $\pgl$-action.  So they descend to the desired maps $s_i$ and $e$.
\end{proof}

\subsection{Gromov-Witten invariants for indecomposable classes}\label{subsec:Indecomposable}
When the spherical homology class associated to a Gromov-Witten invariant is indecomposable, there is an alternative representation of the Gromov-Witten pseudocycle which is a closed manifold.  First, we recall the notion of an indecomposable homology class and then give the construction.

\begin{notn}
Let $(X,\Omega)$ be a closed symplectic manifold.
\end{notn}

\begin{defn}\label{defn:JIndecomposable}
Let $J$ be an $\Omega$-tamed almost complex structure.  A class $A \in H_2(X;\IZ)$ is \emph{$J$-indecomposable}, if $A$ can not be written as a sum $A_1 + A_2$, where $A_i \in H_2(X;\IZ)$ is represented by a possibly disconnected non-constant nodal genus zero $J$-holomorphic curve.
\end{defn}

By \cite[Exercise 7.1.10]{McDuffSalamon_JHolomorphicCurves}, for $A \in H_2(X;\IZ)$, the subspace of all $\Omega$-tamed almost complex structures $J$ for which $A$ is $J$-indecomposable is open.

\begin{defn}\label{defn:OmegaIndecomposable}
A class $A \in H_2(X;\IZ)$ is \emph{$\Omega$-indecomposable}, if $A$ is $J$-indecomposable for some $\Omega$-tamed almost complex structure $J$.
\end{defn}

We now construct the alternative representation of the Gromov-Witten pseudocycle.

\begin{notn}\label{notn:IndecomposableOpenStratum}
Fix an $\Omega$-indecomposable homology class $A \in H_2(X;\IZ)$ and an $\Omega$-tamed almost complex structure $J$ on $X$ for which $A$ is $J$-indecomposable.  Define the spaces
\[ \sM_{0,0}^A(X,J) \coloneqq \{ u: \IP^1 \to X \mid du(z) + J(u(z)) \circ du(z) \circ j(z) = 0,\ [u] = A \} \]
and
\[ \sM_{0,k}^A(X,J) \coloneqq \left( \sM_{0,0}^A(X,J) \times (\IP^1)^k \smallsetminus \Delta_{\fat} \right) / \pgl, \]
where $\pgl$ acts by
\[ \psi \cdot (u,z_1,\dots,z_k) \mapsto (u \circ \psi^{-1},\psi(z_1),\dots,\psi(z_k)). \]
Define
\[ \widehat{\sM}_{0,k}^A(X,J) = \left( \sM_{0,0}^A(X,J) \times (\IP^1)^k \right) / \pgl. \]
There is an evaluation map
\[ \widehat{\ev}:  \widehat{\sM}_{0,k}^A(X,J) \to X^k, \gap \widehat{\ev}([u,z_1,\dots,z_k]) = (u(z_1),\dots,u(z_k)). \]
There is a forgetful map
\[ \pi:  \widehat{\sM}_{0,k+1}^A(X,J) \to  \widehat{\sM}_{0,k}^A(X,J), \gap \pi([u,z_1,\dots,z_k,z_{k+1}]) = [u,z_1,\dots,z_k]. \]
There are sections of $\pi$
\[ s_i : \widehat{\sM}_{0,k}^A(X,J) \to \widehat{\sM}_{0,k+1}^A(X,J), \gap s_i([u,z_1,\dots,z_k]) = [u,z_1,\dots,z_k,z_i]. \]
Notice that $pr_{k+1} \circ \widehat{\ev} \circ s_i = pr_{i} \circ \widehat{\ev}$.
\end{notn}

\begin{rem}
One can take the stable maps compactification of $\sM_{0,k}^A(X,J)$, denoted $\overline{\sM}_{0,k}^A(X,J)$, and show that it is a smooth closed oriented manifold for some regular $J$; however, for our purposes, it is convenient to use the compactification $\widehat{\sM}_{0,k}^A(X,J)$.
\end{rem}

\begin{prop}\label{prop:NaiveCompactificationOfIndecomp}
For a generic choice of $\Omega$-tamed almost complex structure $J$ as in \cref{notn:IndecomposableOpenStratum}:
\begin{enumerate}
	\item The space $\widehat{\sM}_{0,k}^A(X,J)$ is a smooth closed oriented manifold of dimension
	\[ 2n+2k+2c_1(A)-6. \]
	\item The maps $\pi$, $s_{i}$, and $\widehat{\ev}$ are smooth.
	\item The map $\pi$ defines an oriented $S^2$-bundle over $\widehat{\sM}_{0,k}^A(X,J)$.
\end{enumerate}
Finally, the class of $\widehat{\ev}_*[\widehat{\sM}_{0,k}^A(X,J)]$
is independent of the choice of generic $J$ and agrees with the class of the pseudocycle
\[ \ev: \sM_{0,k}^A(X,J) \to X^k, \gap {\ev}([u,z_1,\dots,z_k]) = (u(z_1),\dots,u(z_k)).\footnote{This defines a pseudocycle by \cite[Lemma 7.1.8]{McDuffSalamon_JHolomorphicCurves}.} \]
\end{prop}

\begin{proof}
Every holomorphic curve that represents an indecomposable class is simple.  So \cite[Theorem 3.1.6]{McDuffSalamon_JHolomorphicCurves} implies that for a generic choice of $J$ as in \cref{notn:IndecomposableOpenStratum}, the space $\sM_{0,0}^A(X,J) \times (\IP^1)^k$ is a smooth oriented (non-compact) manifold of dimension $2n+2c_1(A)+2k$.  By \cite[Exercise 6.1.4]{McDuffSalamon_JHolomorphicCurves}, the action of $\pgl$ on $\sM_{0,0}^A(X,J)$ is proper and free and orientation preserving since $\pgl$ is connected.  Also, the action of $\pgl$ on $(\IP^1)^k$ is orientation preserving.  So the action of $\pgl$ on the product $\sM_{0,0}^A(X,J) \times (\IP^1)^k$ is proper and free and orientation preserving on each factor.  By the Slice Theorem \cite[Proposition 2.2.2, Remark 2.2.3, and Section 4.1]{Palais_ExistenceOfSlices}, the quotient $\widehat{\sM}_{0,k}^A(X,J)$ is a smooth oriented manifold of dimension $2n+2k+2c_1(A)-6$.
The quotient is closed by Gromov compactness and the indecomposability of $A$ (cf. \cite[Proof of Lemma 7.1.8]{McDuffSalamon_JHolomorphicCurves}).  This proves item (i).

The maps $\pi$, $s_i$ and $\widehat{\ev}$ are induced from smooth $\pgl$-equivariant maps on the product $\sM_{0,0}^A(X,J) \times (\IP^1)^k$.  By the Slice Theorem \cite[Proposition 2.2.2, Remark 2.2.3, and Section 4.1]
{Palais_ExistenceOfSlices}, the quotient projection
\[ \sM_{0,0}^A(X,J) \times (\IP^1)^k \to \widehat{\sM}_{0,k}^A(X,J) \]
is a submersion.  So by \cite[Theorem 4.2.9]{Lee_IntroductionToSmoothManifolds}, the above maps descend to smooth maps on $\widehat{\sM}_{0,k}^A(X,J)$.  This proves item (ii).  By \cref{lem:QuotientS2BundleOrientation} and our observations, $\pi$ defines an oriented $S^2$-bundle, proving item (iii).

To prove the final conclusion notice that $\sM_{0,k}^A(X,J)$ is an open subset of $\widehat{\sM}_{0,k}^A(X,J)$ and its complement is covered by smooth, closed manifolds of dimensions less than or equal to $2n+2k+2c_1(A)-8$.
By \cref{lem:ExcisingSubmanifoldsFromPseudocycles}, $\widehat{\ev}$ and $\ev$ define the same pseudocycles.  By \cite[Lemma 7.1.8]{McDuffSalamon_JHolomorphicCurves}, the pseudocycle $\ev$ is independent of the choice of $J$.  So $\widehat{\ev}$ is also independent of the choice of $J$.  This completes the proof of the proposition.
\end{proof}

\begin{rem}\label{rem:GWInvariantsAgree}
By \cite[Remark 1.6]{CieliebakMohnke_HypersurfacesAndTransversality}, the pseudocycle $\widehat{\ev}$ from \cref{prop:NaiveCompactificationOfIndecomp} can be used instead of a pseudocycle $\ev_{\ell+k}(X,\Omega;\omega,\mathbf{J},Y)$ from \cref{defn:GWInvariants} to compute the Gromov-Witten invariants of $X$ for the an $\Omega$-indecomposable class $A$.
\end{rem}


\section{Pseudocycles and the intersection product}\label{sec:Pseudocycles}

The goal of this section is to define an intersection product for pseudocycle homology and establish some of its basic properties that we will need to prove our main result.  Throughout this entire section, we fix the following notation.

\begin{notn}
Let $X$ be a smooth manifold.
\end{notn}

\subsection{Definitions and the singular homology isomorphism}
We review the definition of a pseudocycle and the isomorphism relating pseudocycle homology with singular homology.

\begin{defn}\label{defn:Pseudocycle}
A smooth map $f: M \to X$ is a \emph{$k$-pseudocycle} if
\begin{enumerate}
	\item $M$ is a smooth oriented $k$-dimensional manifold,
	\item $\overline{f(M)}$ is compact, and
	\item the dimension of
	\[ \Bd(f) \coloneqq  \bigcap_{\begin{matrix} K \subset M \\ \mbox{compact} \end{matrix}} \overline{f(M \smallsetminus K)} \]
	is at most $k-2$, that is, $\Bd(f)$ is covered by the image of a smooth map $g: W \to X$, where $W$ is a smooth manifold of dimension $k-2$.
\end{enumerate}
Two $k$-pseudocycles $f_0: M_0 \to X$ and $f_1: M_1 \to X$ are \emph{equivalent} if there exists a smooth map $F: B \to X$ that satisfies:
\begin{enumerate}
	\item $B$ is a smooth oriented $(k+1)$-dimensional manifold with boundary,
	\item $\overline{F(B)}$ is compact,
	\item the dimension of $\Bd(F)$ is at most $k-1$,
	\item $\partial B = M_1 - M_0$, and
	\item $F|_{M_i} = f_i$ for $i=0,1$.
\end{enumerate}
We write $f_0 \sim f_1$ to denote an equivalence and write $[f]$ to denote an equivalence class.
\end{defn}

The equivalence classes of $k$-pseudocycles define a $\IZ$-module under disjoint unions, denoted $\sH_k(X)$.  Inverses are given by reversing orientations of pseudocycles.

\begin{thm}\label{thm:ZingerIsomorphism}\cite[Theorem 1.1]{Zinger_PseudocyclesAndIntegralHomology}
There is a natural isomorphism
\[ \Phi_\bullet: \sH_\bullet(X) \to H_\bullet(X;\IZ).\]
\end{thm}

To define the map $\Phi$ in \cref{thm:ZingerIsomorphism}, one uses the following observation due to Zinger.

\begin{prop}\label{prop:ZingerNeighborhood}\cite[Proposition 2.2]{Zinger_PseudocyclesAndIntegralHomology}
Let $W$ be a smooth $k$-dimensional manifold and let $g: W \to X$ be a smooth map.  There is an open subset $U \subset X$ such that $g(W) \subset U$ and
\[ H_\bullet(U;\IZ) = 0 \gap \mbox{for} \gap \bullet > \dim(W). \]
\end{prop}

\cref{prop:ZingerNeighborhood} follows from \cref{lem:ZingerCover} and an inductive argument using Mayer-Vietoris.

\begin{lem}\label{lem:ZingerCover}\cite[Proof of Lemma 2.4]{Zinger_PseudocyclesAndIntegralHomology}
Let $W$ be a smooth $k$-dimensional manifold and let $g: W \to X$ be a smooth map.  For $0 \leq i \leq k$, there exist open subsets $U_i \subset X$ such that each intersection
\[ U_{i_1} \cap \cdots \cap U_{i_\ell} \gap \mbox{for} \gap 0 \leq i_j \leq k \]
is a disjoint union of contractible open subsets and $g(W) \subset \cup_i U_i$.
\end{lem}

\begin{con}\label{con:ZingerMap}\cite[Lemma 3.5]{Zinger_PseudocyclesAndIntegralHomology}
Given a $k$-pseudocycle $f: M \to X$, the element $\Phi(f) \in H_k(X;\IZ)$ is defined as follows.  Suppose that $g: W \to X$ covers $\Bd(f)$.   By \cref{prop:ZingerNeighborhood}, there exists an open neighborhood $U$ of $\Ima(g) \subset X$ such that $H_\bullet(U;\IZ) = 0$ for $\bullet > k-2$.  Define $V = f^{-1}(U)$.  Fix a compact codimension zero submanifold with boundary $M_0 \subset M$ such that $\partial M_0 \subset V$ and $M \smallsetminus V \subset M_0$.  Orient $M_0$ with respect to the orientation on $M$.  The map $f$ induces a map of pairs $f: (M_0,\partial M_0) \to (X,U)$.  The homology class associated to the pseudocycle $f$ is the image of the (relative) fundamental class of $(M_0,\partial M_0)$ under the composition
\[ \xymatrix{ H_k(M_0,\partial M_0;\IZ) \ar[r]^f & H_k(X,U;\IZ) & H_k(X;\IZ), \ar[l]_{\cong}} \]
where the right isomorphism follows from the vanishing of $H_\bullet(U;\IZ)$ for $\bullet> k-2$.
\end{con}

The following fact will be used implicitly throughout the next subsection.

\begin{lem}\cite[Lemma 6.5.9]{McDuffSalamon_JHolomorphicCurves}
If $f: M \to X$ is a pseudocycle and $\psi: X \to Y$ is a smooth map of manifolds, then $\psi \circ f$ is a pseudocycle.  Moreover, if $f_0$ is equivalent to $f_1$, then $\psi \circ f_0$ is equivalent to $\psi \circ f_1$.\qed
\end{lem}

Finally, we remark that removing a codimension two submanifold from a pseudocycle does not change the equivalence class.

\begin{lem}\label{lem:ExcisingSubmanifoldsFromPseudocycles}
If $f: M \to X$ is a $k$-pseudocycle and $N \subset M$ is a submanifold of codimension at least two, then $f|_{M \smallsetminus N}: M \smallsetminus N \to X$ is a pseudocycle that is equivalent to $f$.
\end{lem}

\begin{proof}
To see that $f|_{M \smallsetminus N}$ is a pseudocycle, notice that $\Bd(f|_{M \smallsetminus N}) \subset \Bd(f) \cup f(N)$ and $\overline{f|_{M \smallsetminus N}(M \smallsetminus N)} \subset \overline{f(M)}$.  So $f|_{M \smallsetminus N}$ is precompact and its boundary has dimension at most $k-2$.  So it is a pseudocycle.  To obtain the equivalence, define $B = I \times M \smallsetminus \{0\} \times N$ and
\[ F: B \to X, \gap F(t,p) = f(p). \]
Notice that $\Bd(F) \subset \Bd(f) \cup f(N)$.
So the dimension of $\Bd(F)$ is at most $k-1$.  The remaining conditions for $F$ to define an equivalence from $f$ to $f|_{M \smallsetminus N}$ are clear.
\end{proof}

\subsection{Intersections of pseudocycles}

The goal of this subsection is to define an intersection product for pseudocycles and show that it is well-defined.  We begin by fixing an appropriate notion of transversality for pseudocycles.

\begin{defn}
Two pseudocycles $f_0: M_0 \to X$ and $f_1: M_1 \to X$ are \emph{strongly transverse} if there exist coverings $g_i: W_i \to X$ of the $\Bd(f_i)$ such that  
\begin{enumerate}
	\item $f_0$ is transverse to $f_1$,\footnote{Recall, that two maps $h_i: N_i \to X$ for $i = 0,1$ are transverse, if for each pair of points $(q_0,q_i) \in N_0 \times N_1$ such that $h_0(q_0) = x = h_1(q_1)$, the operator $d(h_0)_{q_0} + d(h_1)_{q_1}$ surjects onto $T_xX$.}
	\item $g_0$ is transverse to $g_1$,
	\item $f_0$ is transverse to $g_1$, and
	\item $g_0$ is transverse to $f_1$.
\end{enumerate}
\end{defn}

We can take the intersection of strongly transverse pseudocycles.

\begin{lem}\label{lem:IntersectingPseudocycles}
If $f_i: M_i \to X$ are strongly transverse $k_i$-pseudocycles, then
\[ (f_0 \cdot f_1) \coloneqq f: M \to X, \gap M \coloneqq \left\{ (p_0,p_1) \in M_0 \times M_1 \mid f_0(p_0) = f_1(p_1) \right\} \]
given by $f \coloneqq f_i \circ pr_i$ is a $(k_0+k_1-\dim(X))$-pseudocycle.
\end{lem}

\begin{proof}
We verify the conditions of \cref{defn:Pseudocycle}.  First, by transversality, $M$ is a smooth oriented manifold of dimension $k_0+k_1-\dim(X)$ and the map $f$ is smooth.  Second,
\[ \overline{f(M)} \subset \overline{f_0(M_0)} \cup \overline{f_1(M_1)}.\]
So $\overline{f(M)}$ is compact.  Third, we claim that
\[ \Bd(f) \subset \left( \Bd(f_0) \cap f_1(M_1) \right) \bigcup \left(f_0(M_0) \cap \Bd(f_1) \right) \bigcup \left( \Bd(f_0) \cap \Bd(f_1) \right). \]
To see this, let $K_{i,j}$ for $j \in \IN$ be a sequence of compact subsets that exhaust $M_i$.  Then
\[ K_j = M \cap (K_{0,j} \times K_{1,j}) \]
is a sequence of compact subsets of $M$ that exhaust $M$.  Since
\[ M \smallsetminus K_j = M \cap \left( \left( (M_0 \smallsetminus K_{0,j}) \times M_1 \right) \bigcup \left( M_0 \times ( M_1 \smallsetminus K_{1,j}) \right) \right), \]
the above claim follows.  By transversality, the following subsets are smooth manifolds of dimensions less than or equal to $k_0+k_1-\dim(X)-2$:
\begin{align*}
Z_{01} & = \{ (w_0,w_1) \in W_0 \times W_1 \mid g_0(w_0) = g_1(w_1) \},
\\ Z_{1} & = \{ (p_0,w_1) \in M_0 \times W_1 \mid f_0(p_0) = g_1(w_1) \}, \mbox{ and} 
\\ Z_{0} & = \{ (w_0,p_1) \in W_0 \times M_1 \mid g_0(p_0) = f_1(w_1) \}.
\end{align*}
Moreover, their images under $g_0 \circ pr_0$, $f_0 \circ pr_0$, and $g_0 \circ pr_0$ respectively cover $\Bd(f)$.  So the dimension of $\Bd(f)$ is at most $k_0+k_1-\dim(X)-2$.  This completes the verification of the conditions of \cref{defn:Pseudocycle}.
\end{proof}

Intersecting strongly transverse pseudocycles induces an intersection product on $\sH_\bullet(X)$:

\begin{prop}\label{prop:IntersectionProduct}
There exists a well-defined \emph{intersection product} on $\sH_\bullet(X)$ given by the $\IZ$-bilinear map
\[ (-) \cdot (-): \sH_k(X) \times \sH_\ell(X) \to \sH_{k+\ell- \dim(X)}(X), \gap [f_0] \cdot [f_1] \mapsto [\widetilde{f_0} \cdot \widetilde{f_1}], \]
where $\widetilde{f_i} \in [f_i]$ and $\widetilde{f_0}$ is strongly transverse to $\widetilde{f_1}$.
\end{prop}

To prove \cref{prop:IntersectionProduct}, we need to show that the perturbations $\widetilde{f_i}$ exist and that the equivalence class of the intersection in \cref{prop:IntersectionProduct} is independent of all choices.  This amounts to fairly standard transversality arguments.  For completeness, we give them here.  We begin with reviewing transversality for smooth maps, see \cite{GuilleminPollack_DifferentialTopology} and \cite[Proof of Lemma 6.5.5]{McDuffSalamon_JHolomorphicCurves} for more details.

\begin{notn}
Let $\Diff_0^+(X)$ denote the path component of the group of orientation preserving diffeomorphisms of $X$ that contains the identity.  Let $\sP_{\phi} \subset C^{\infty}(I,\Diff_0^+(X))$ denote the space of diffeotopies of $X$ from the identity to $\phi \in \Diff_0^+(X)$.
\end{notn}

\begin{rem}\label{rem:TransversalityForMaps}
Fix smooth maps $h_i: Z_i \to X$ for $i=0,1$.  Consider the evaluation map
\[ \ev: X \times \Diff_0^+(X) \to X, \gap \ev(x,\phi) = \phi(x). \]
The composition $\ev \circ (h_1 \times \Ione)$ is transverse to $h_0$.  So by Sard's Theorem, there is a residual subset $\sF(X;h_0,h_1) \subset \Diff_0^+(X)$ such that for $\phi \in \sF(X;h_0,h_1)$, $h_0$ and $\phi \circ h_1$ are transverse.  Similarly, consider the evaluation map
\[ \ev: I \times X \times \sP_\phi \to X, \gap \ev(t,x,\Phi) = \Phi(t,x).\]
If $h_0$ is transverse to both $h_1$ and $\phi \circ h_1$, then the composition $\ev \circ (\Ione \times h_0 \times \Ione)$ is transverse to $h_0$.  So by Sard's Theorem, there is a residual subset $\sF_\phi(X;h_0,h_1) \subset \sP_\phi$ such that for $\Phi \in \sF_\phi(X;h_0,h_1)$, $h_0$ and $\Phi \circ h_1$ are transverse.
\end{rem}

\cref{prop:IntersectionProduct} will following from \cref{lem:IntersectionLemmas}.

\begin{lem}\label{lem:IntersectionLemmas}
Consider pseudocycles $f_i: M_i \to X$ for $i=0,1,2$.
\begin{enumerate}
	\item \label{lem:GenericTransversality}There exists a residual subset $\sF^{reg}(X;f_0,f_1) \subset \Diff_0^+(X)$ such that for each $\phi \in \sF^{reg}(X;f_0,f_1)$, the pseudocycles $f_0$ and $\phi \circ f_1$ are strongly transverse.
	\item \label{lem:FlowInvariance}Given $\phi \in \Diff_0^+(X)$, the pseudocycle $\phi \circ f_i$ is equivalent to $f_i$.
	\item \label{lem:FlowInvarianceOfIntersection}Suppose that $\phi \in \Diff_0^+(X)$ such that $f_0$ and $\phi \circ f_0$ are both strongly transverse to $f_1$.  Then $f_0 \cdot f_1$ is equivalent to $(\phi \circ f_0) \cdot f_1$.
	\item \label{lem:EquivalenceToEquivalentIntersections}If $f_0$ is equivalent to $f_2$, then there exists $\phi \in \Diff_0^+(X)$ such that $f_0$ and $f_2$ are both strongly transverse to $\phi \circ f_1$ and $f_0 \cdot (\phi \circ f_1)$ is equivalent to $f_2 \cdot (\phi \circ f_1)$.
\end{enumerate}
\end{lem}

Assuming \cref{lem:IntersectionLemmas}, we prove \cref{prop:IntersectionProduct}.

\begin{proof}[Proof of \cref{prop:IntersectionProduct}]
Given classes $[f_i] \in \sH_{k_i}(X)$, by \cref{lem:IntersectionLemmas} \cref{lem:GenericTransversality} and \cref{lem:FlowInvariance}, there exists $\widetilde{f_0} \in [f_0]$ and $\wt{f_1} \in [f_1]$ such that $\widetilde{f_0}$ is strongly transverse to $\widetilde{f_1}$.  Define
\[ [f_0] \cdot [f_1] = [\wt{f_0} \cdot \widetilde{f_1}].\]
Suppose that $h_0 \in [f_0]$ and $h_1 \in [f_1]$ such that $h_0$ is strongly transverse to $h_1$.  We need to show that $\wt{f_0} \cdot \wt{f_1}$ is equivalent to $h_0 \cdot h_1$.  By \cref{lem:IntersectionLemmas} \cref{lem:FlowInvarianceOfIntersection} and \cref{lem:EquivalenceToEquivalentIntersections}, there exist $\phi,\psi \in \Diff_0^+(X)$ such that $\wt{f_0}$ and $h_0$ are both strongly transverse to $\phi \circ \wt{f_1}$, such that $\phi \circ \wt{f_1}$ and $h_1$ are both strongly transverse to $\psi \circ h_0$, and such that
\[ [\wt{f_0} \cdot \wt{f_1}] = [\wt{f_0} \cdot (\phi \circ \wt{f_1})] = [h_0 \cdot (\phi \circ \wt{f_1})] = [(\psi \circ h_0) \cdot (\phi \circ \wt{f_1})] = [(\psi \circ h_0) \cdot h_1] = [h_0 \cdot h_1]. \]
This completes the proof.
\end{proof}

We now prove each item in \cref{lem:IntersectionLemmas} in turn.

\begin{proof}[Proof of \cref{lem:IntersectionLemmas} \cref{lem:GenericTransversality}]
If $g_i: W_i \to X$ covers $\Bd(f_i)$, then using \cref{rem:TransversalityForMaps} set
\[ \sF^{reg}(X;f_0,f_1) \coloneqq \sF(X;f_0,f_1) \cap \sF(X;f_0,g_1) \cap \sF(X;g_0,f_1) \cap \sF(X;g_0,g_1). \]
\end{proof}

\begin{proof}[Proof of \cref{lem:IntersectionLemmas} \cref{lem:FlowInvariance}]
Suppose $g: W \to X$ covers $\Bd(f)$.  By definition, there exists a diffeotopy $\phi_t: I \times X \to X$ such that $\phi_0 = \Ione$ and $\phi_1 = \phi$.  Consider
\[ F: I \times M \to X, \gap  F(t,p) = \phi_t(f(p)) \]
and
\[ G: I \times W \to X, \gap G(t,w) = \phi_t(g(w)).\]
One can check that $G$ covers $\Bd(F)$ and that $F$ gives an equivalence between $f$ and $\phi \circ f$.
\end{proof}

\begin{proof}[Proof of \cref{lem:IntersectionLemmas} \cref{lem:FlowInvarianceOfIntersection}]
If $g_i: W_i \to X$ cover $\Bd(f_i)$, then using \cref{rem:TransversalityForMaps} fix
\[ \Phi \in \sF_\phi(X;f_1,f_0) \cap \sF_\phi(X;f_1,g_0) \cap \sF_\phi(X;g_1,f_0) \cap \sF_\phi(X;g_1,g_0). \]
Define
\[ \widetilde{F}: I \times M_0 \to X, \gap \widetilde{F}(t,p) = \Phi(t,f_0(p))\]
and
\[ \wt{G}: I \times W_0 \to X, \gap \wt{G}(t,w) = \Phi(t,g_0(w)).\]
By construction, $\wt{F}$ and $\wt{G}$ are both transverse to both $f_1$ and $g_1$.  By transversality, 
\[ B' \coloneqq \{ (t,p_0,p_1) \in I \times M_0 \times M_1 \mid \wt{F}(t,p_0) = f_1(p_1) \} \]
is a smooth oriented $(k_0 + k_1 +1 - \dim(X))$-dimension manifold with boundary such that $\partial B' = \wt{M} - M$,
where ${M}$ and $\wt{M}$ are the domains of the intersections $f_0 \cdot f_1$ and $(\phi \circ f_0) \cdot f_1$ respectively.  Moreover, $F' \coloneqq f_1 \circ pr_1: B' \to X$ is smooth and satisfies
\[ F'|_{\wt{M}} = \phi \circ f_0 \gap \mbox{and} \gap F'|_{M} = f_0. \]
Since $\Ima(F') \subset \Ima(f_1)$, $\overline{F'(B')} \subset \overline{f_1(M_1)}$ is compact.  To conclude that $F': B' \to X$ is an equivalence from $f_0 \cdot f_1$ to $(\phi \circ f_0) \cdot f_1$, we need that $\Bd(F')$ has dimension at most $k_0 + k_1 -1- \dim(X)$.  As in the proof of \cref{lem:IntersectingPseudocycles},
\[ \Bd(F') \subset \left( \Bd(\wt{F}) \cap f_1(M_1) \right) \bigcup \left(\wt{F}(I \times M_0) \cap \Bd(f_1) \right) \bigcup \left( \Bd(\wt{F}) \cap \Bd(f_1) \right). \]
Now, by transversality, the following are smooth manifolds with boundary of dimensions less than or equal to $k_0+k_1-1-\dim(X)$:
\begin{align*}
Z_{01} & = \{ (t,w_0,w_1) \in I \times W_0 \times W_1 \mid \wt{G}(t,w_0) = g_1(w_1) \},
\\ Z_{1} & = \{ (t,p_0,w_1) \in I \times M_0 \times W_1 \mid \wt{F}(t,p_0) = g_1(w_1) \}, \mbox{ and} 
\\ Z_{0} & = \{ (t,w_0,p_1) \in I \times W_0 \times M_1 \mid \wt{G}(t,w_0) = f_1(w_1) \}.
\end{align*}
Their images under the smooth mappings $\wt{G} \circ pr_0$, $\wt{F} \circ pr_0$, and $\wt{G} \circ pr_0$ respectively cover $\Bd(F')$.  So the dimension of $\Bd(F')$ is at most $k_0+k_1-1-\dim(X)$, as desired.
\end{proof}

\begin{proof}[Proof of \cref{lem:IntersectionLemmas} \cref{lem:EquivalenceToEquivalentIntersections}]
Let $F: B \to X$ be a smooth map that realizes the equivalence between $f_0$ and $f_2$.  Suppose that $G: V \to X$ covers $\Bd(F)$.  Suppose that $g_0: W_0 \to X$ and $g_2: W_2 \to X$ cover $\Bd(f_0)$ and $\Bd(f_2)$ respectively.  Using \cref{rem:TransversalityForMaps}, define
\begin{align*}
\sF \coloneqq \sF(X;F,f_1) & \cap \sF(X;F,g_1) \cap \sF(X;G,f_1) \cap \sF(X;G,g_1)
\\ & \cap \sF(X;f_0,f_1) \cap \sF(X;f_0,g_1) \cap \sF(X;g_0,f_1) \cap \sF(X;g_0,g_1)
\\ & \cap \sF(X;f_2,f_1) \cap \sF(X;f_2,g_1) \cap \sF(X;g_2,f_1) \cap \sF(X;g_2,g_1).
\end{align*}
Fix $\phi \in \sF$.  We have that $f_0$ is strongly transverse to $\phi \circ f_1$ and $f_2$ is strongly transverse to $\phi \circ f_1$.  To complete the proof, we need to show that $f_0 \cdot (\phi \circ f_1)$ is equivalent to $f_2 \cdot (\phi \circ f_1)$.  

Fix domains $M_{01}$ and $M_{21}$ for $f_0 \cdot (\phi \circ f_1)$ and $f_2 \cdot (\phi \circ f_1)$ respectively.  Define
\[ \wt{F} \coloneqq \phi \circ f_1 \circ pr_1: \wt{B} \to X, \gap \wt{B} \coloneqq \{ (b,p_1) \in B \times M_1 \mid F(b) = \phi \circ f_1(p_1) \}. \]
By transversality, $\wt{B}$ is a smooth oriented $(k_0+k_1+1 - \dim(X))$-dimensional manifold with boundary such that $\partial \wt{B} = M_{01} - M_{21}$.
Moreover,
\[ \wt{F}|_{M_{01}} = f_0 \gap \mbox{and} \gap \wt{F}|_{M_{21}} = f_2.\]
Since $\Ima(\wt{F}) \subset \Ima(\phi \circ f_1)$, $\overline{\wt{F}(\wt{B})}$ is compact.  To conclude that $\wt{F}$ gives an equivalence from $f_0 \cdot (\phi \circ f_1)$ to $f_2 \cdot (\phi \circ f_1)$, we need that $\Bd(\wt{F})$ has dimension at most $(k_0 + k_1 -1 -\dim(X))$.  To see this, one argues identically to the argument at the end of \cref{lem:IntersectionLemmas} \cref{lem:FlowInvarianceOfIntersection}.
\end{proof}

\subsection{Properties of intersections on product manifolds}

Suppose that $f_i: M_i \to X_i$ are $k_i$-pseudocycles.  The product $f_0 \times f_1: M_0 \times M_1 \to X_0 \times X_1$ is a $(k_0+k_1)$-pseudocycle.  The main property of the intersection product that we will need is the following.

\begin{prop}\label{prop:KunnethForPseudocycles}
Given pseudocycles $h_i: N_i \to X_i$ for $i = 0,1$ and a pseudocycle $f: M \to X_0 \times X_1$, the intersection product satisfies:
\[ [f] \cdot [h_0 \times h_1] = [pr_0 \circ ([f] \cdot [\Ione \times h_1])] \cdot [h_0]. \]
\end{prop}

\cref{prop:KunnethForPseudocycles} will follow from \cref{lem:FactorTransverseImpliesIoneTransverse} and \cref{lem:FactorTransversality} below.

\begin{lem}\label{lem:FactorTransverseImpliesIoneTransverse}
Suppose that $f: M \to X_0 \times X_1$ and $h: N \to X_1$ are pseudocycles.  If $pr_1 \circ f$ is strongly transverse to $h$, then $f$ is strongly transverse to $\Ione \times h$.
\end{lem}

\begin{proof}
Suppose that $\Bd(f)$ is covered by $g: W \to X_0 \times X_1$ and $\Bd(h)$ is covered by $r: V \to X_1$.  It follows that $\Bd(\Ione \times h)$ is covered by $\Ione \times r: X_0 \times V \to X_0 \times X_1$.  We need to show that $f \pitchfork \Ione \times h$, $f \pitchfork \Ione \times r$, $g \pitchfork \Ione \times h$, and $g \pitchfork \Ione \times r$.  We will show that $f \pitchfork \Ione \times h$ (the others being entirely analogous).  

Write $f$ as $f = (f_0,f_1)$, where $f_i = pr_i \circ f$.  We need to show that for each
\[ (p,x_0,q) \in M \times X_0 \times N \gap \mbox{such that} \gap f(p) = (x_0,x_1) = (x_0,h(q)), \]
the operator
\[ (d(f_0)_p + \Ione, d(f_1)_p + dh_q): T_p M \oplus T_{x_0} X_0 \oplus T_q N \to T_{x_0}X_0 \oplus T_{x_1}X_1 \]
is surjective.  Consider a vector $(v_0,v_1) \in T_{x_0}X_0 \oplus T_{x_1}X_1$.  Since $f_1$ is transverse to $h$, there exists $\xi \in T_pM$ and $\zeta_1 \in T_q N$ such that 
\[ v_1 = d(f_1)_p(\xi) + dh_q(\zeta_1).\]
So
\[ (v_0,v_1) = (d(f_0)_p + \Ione, d(f_1)_p + dh_q)(\xi,v_0 - d(f_0)_p(\xi),\zeta_1).\]
\end{proof}

\begin{lem}\label{lem:FactorTransversality}
Given pseudocycles $h_i: N_i \to X_i$ and a pseudocycle $f: M \to X_0 \times X_1$, there exist $\phi_i \in \Diff_0^+(X_i)$ such that
\begin{enumerate}
	\item $pr_1 \circ f$ is strongly transverse to $\phi_1 \circ h_1$,
	\item $pr_0 \circ ( f \cdot (\Ione \times (\phi_1 \circ h_1)))$ is strongly transverse to $\phi_0 \circ h_0$, and
	\item $f$ is strongly transverse to $(\phi_0 \circ h_0) \times (\phi_1 \circ h_1)$.
\end{enumerate}
\end{lem}

\begin{proof}
Consider the pseudocycle $pr_1 \circ f: M \to X_1$.  By \cref{lem:IntersectionLemmas} \cref{lem:GenericTransversality}, there exists $\phi_1 \in \Diff_0^+(X_1)$ such that $pr_1 \circ f$ is strongly transverse to $\phi_1 \circ h_1$.  This proves item (i).

By \cref{lem:FactorTransverseImpliesIoneTransverse}, $f$ is strongly transverse to $\Ione \times (\phi_1 \circ h_1)$.  Let $e: P \to X_0 \times X_1$ denote the pseudocycle given by $f \cdot (\Ione \times (\phi_1 \circ h_1))$.  By \cref{lem:IntersectionLemmas} \cref{lem:GenericTransversality}, there exists $\phi_0 \in \Diff_0^+(X_0)$ such that $pr_0 \circ e$ is strongly transverse to $\phi_0 \circ h_0$.  This proves item (ii).

We now show that $f$ is strongly transverse to $\phi_0 \circ h_0 \times \phi_1 \circ h_1$.  Suppose that $\Bd(f)$ was covered by $g: W \to X_0 \times X_1$.  Suppose that $\Bd(h_i)$ was covered by $r_i: V_i \to X_i$.  To show strong transversality, we need to show that $f \pitchfork (\phi_0 \circ h_0) \times (\phi_1 \circ h_1)$, $f \pitchfork (\phi_0 \circ h_0) \times (\phi_1 \circ r_1)$, $f \pitchfork (\phi_0 \circ r_0) \times (\phi_1 \circ h_1)$, $f \pitchfork (\phi_0 \circ r_0) \times (\phi_1 \circ r_1)$, $g \pitchfork (\phi_0 \circ h_0) \times (\phi_1 \circ h_1)$, $g \pitchfork (\phi_0 \circ h_0) \times (\phi_1 \circ r_1)$, $g \pitchfork (\phi_0 \circ r_0) \times (\phi_1 \circ h_1)$, and $g \pitchfork (\phi_0 \circ r_0) \times (\phi_1 \circ r_1)$.
We will prove the first (the others being entirely analogous).  Write $f$ as $f = (f_0,f_1)$, where $f_i = pr_i \circ f$.  We need to show that for each
\[ (p,q_0,q_1) \in M \times N_0 \times N_1 \gap \mbox{such that} \gap f(p) = (\phi_0 \circ h_0(q_0),\phi_1 \circ h_1(q_1)) \]
that
\[ (d(f_0)_p + d(\phi_0 \circ h_0)_{q_0}, d(f_1)_p + d(\phi_1 \circ h_1)_{q_1}) : T_p M \oplus T_{q_0} N_0 \oplus T_{q_1} N_1 \to T_{x_0} X_0 \oplus T_{x_1} X_1 \]
is surjective.  Fix $(v_0,v_1) \in T_{x_0}X_0 \oplus T_{x_1}X_1$.  First, consider $(v_0,0)$.  By transversality, $T_{(p,x_0,q_1)} P$ is
\[ \{ (\xi,w_0,\zeta_1) \in T_p M \oplus T_{x_0} X_0 \oplus T_{q_1} N_1 \mid (d(f_0)_p(\xi)+w_0, d(f_1)_p(\xi)+d(\phi_1 \circ h_1)_{q_1}(\zeta_1) ) = (0,0)\}.\]
Since $pr_0 \circ e$ is transverse to $\phi_0 \circ h_0$, there exists $(\xi,w_0,\zeta_1) \in T_{(p,q_0,q_1)} P$ and $\zeta_0 \in T_{q_0}N_0$ such that
\[ d(pr_0 \circ e)(\xi,w_0,\zeta_1) + d(\phi_0 \circ h_0)_{q_0}(\zeta_0) = v_0. \]
Since $pr_0 \circ e = f_0 \circ pr_M|_{P}$, we have that
\[ d(f_0)_p(\xi)+d(\phi_0 \circ h_0)(\zeta_0) = v_0. \]
But
\[ d(f_1)_p(\xi)+d(\phi_1 \circ h_1)_{q_1}(\zeta_1) = 0. \]
So
\[ (d(f_0)_p(\xi)+d(\phi_0 \circ h_0)(\zeta_0),d(f_1)_p(\xi)+d(\phi_1 \circ h_1)_{q_1}(\zeta_1)) = (v_0,0). \]

Now consider $(0,v_1)$.  By the above work, it suffices to show that there exists $\xi \in T_pM$ and $\zeta_i \in T_{q_i}N_i$ such that
\[ (d(f_0)_p(\xi) + d(\phi_0 \circ h_0)_{q_0}(\zeta_0), d(f_1)_p(\xi) + d(\phi_1 \circ h_1)_{q_1}(\zeta_1)) = (*,v_1), \]
where $*$ can be any vector.  Since $pr_1 \circ f$ is transverse to $\phi_1 \circ h_1$, we have that there exists $\xi \in T_pM$ and $\zeta_1 \in T_{q_1}N_1$ such that
\[ d(f_1)_p(\xi) + d(\phi_1 \circ h_1)_{q_1}(\zeta_1) = v_1. \]
So
\[ (d(f_0)_p(\xi_0) + d(\phi_0 \circ h_0)_{q_0}(0), d(f_1)_p(\xi_0) + d(\phi_1 \circ h_1)_{q_1}(\zeta_1)) = (*,v_1), \]
as desired.
\end{proof}

We can now prove \cref{prop:KunnethForPseudocycles}.

\begin{proof}[Proof of \cref{prop:KunnethForPseudocycles}]
Consider the $\phi_i \in \Diff_0^+(X_i)$ given by \cref{lem:FactorTransversality}.  Tautologically, we have an equality of pseudocycles:
\[ f \cdot (\phi_0 \circ h_0 \times \phi_1 \circ h_1) = pr_0 \circ (f \cdot (\Ione \times \phi_1 \circ h_1)) \cdot \phi_0 \circ h_0. \]
Applying \cref{lem:FactorTransversality} and \cref{prop:IntersectionProduct}, we have that
\begin{align*}
[f] \cdot [h_0 \times h_1] & = [f \cdot (\phi_0 \circ h_0 \times \phi_1 \circ h_1)] 
\\& = [pr_0 \circ (f \cdot (\Ione \times \phi_1 \circ h_1)) \cdot \phi_0 \circ h_0]
\\ & = [ pr_0 \circ (f \cdot (\Ione \times \phi_1 \circ h_1))] \cdot [h_0] 
\\ & = [ pr_0 \circ ( [f] \cdot [\Ione \times h_1])] \cdot [h_0].
\end{align*}
\end{proof}


\section{Some algebraic topology lemmas}\label{sec:EasyAT}

In this section, we collect some lemmas that are used to establish our main result.

\begin{lem}\label{lem:PositiveDegreeAndSurj}
Let $X$ be a topological space.  Let $Y$ be a closed oriented connected $n$-dimensional manifold.  Suppose that $f: X \to Y$ is a continuous map such that
\[ f_*: H_{n}(X;\IZ) \to H_{n}(Y;\IZ)\]
is non-zero.  The map
\[ f_*: H_\bullet(X;\IQ) \to H_\bullet(Y;\IQ)\]
is surjective.
\end{lem}

\begin{proof}
Let $A \in H_\bullet(Y;\IQ)$.  By Poincar\'{e} duality, there exists $\alpha \in H^{n -\bullet}(Y;\IQ)$ such that $A = [Y] \cap \alpha$.  Since $f_*$ is non-zero on $H_n(-;\IZ)$, there is a $C \in H_{n}(X;\IQ)$ such that $f_*(C) = [Y]$.  By naturality, $f_*(C \cap f^*(\alpha)) = A$.  So $f_*$ is surjective on rational homology.
\end{proof}

\begin{lem}\label{lem:PseudoFiniteIndex}
Let $Y$ be a closed oriented connected $n$-dimensional manifold.  Suppose that $f: X \to Y$ is an $n$-pseudocycle whose class in $\sH_n(Y)$ is non-zero.  The image of
\[ f_*: \pi_1(X) \to \pi_1(Y)\]
is a finite index subgroup of $\pi_1(Y)$.
\end{lem}

\begin{proof}
Let $\pi: \widetilde{Y} \to Y$ denote the covering space associated to the subgroup $\Ima(f_*) \subset \pi_1(Y)$ \cite[Proposition 1.36]{Hatcher_AT}.  The covering degree of $\pi$ is given by the index of the subgroup $\Ima(f_*) \subset \pi_1(Y)$ \cite[Proposition 1.32]{Hatcher_AT}.  Notice that the covering degree of $\pi$ is finite if and only if $\widetilde{Y}$ is compact; however, $\widetilde{Y}$ is orientable, so $\widetilde{Y}$ is compact if and only if $H_{n}(\widetilde{Y};\IZ) \neq 0$.  So to prove the lemma, it suffices to show that $H_{n}(\widetilde{Y};\IZ) \neq 0$.  To see this, by homotopy lifting \cite[Proposition 1.33]{Hatcher_AT}, we have a commutative diagram:
\[ \xymatrix{ & \widetilde{Y} \ar[d]^\pi \\ X \ar[ru]^{\widetilde{f}} \ar[r]^f & Y. \\ } \]
By \cref{lem:ZingerCover}, there exist open subsets $U_i\subset Y$ for $0 \leq i \leq n-2$, such that each intersection
\[ U_{i_1} \cap \cdots \cap U_{i_\ell} \gap \mbox{for} \gap 0 \leq i_j \leq n-2 \]
is a disjoint union of contractible open subsets and
\[ \mbox{Bd}(f) \subset \bigcup_i U_i.\]
Define $\widetilde{U_i} = \pi^{-1}(U_i)$.  Since $\pi$ is a covering map, for $0 \leq i \leq n-2$, each intersection
\[ \wt{U}_{i_1} \cap \cdots \cap \wt{U}_{i_\ell} \gap \mbox{for} \gap 0 \leq i_j \leq n-2 \]
is a disjoint union of contractible open subsets.  So by an inductive argument with Mayer-Vietoris,
\[ H_\bullet(\widetilde{U_i};\IZ) = 0 = H_\bullet(U;\IZ),\gap \mbox{for} \gap \bullet > 0.\]

Let $X_0$ be an oriented codimension zero submanifold of $X$ with boundary as in \cref{con:ZingerMap}.  We have a commutative diagram:
\[ \xymatrix{ & H_n(\widetilde{Y},\widetilde{U};\IZ) \ar[d]^\pi & H_n(\widetilde{Y};\IZ) \ar[d]^\pi \ar[l]^\cong \\ H_n(X_0,\partial X_0;\IZ) \ar[r]^f \ar[ur]^{\widetilde{f}} & H_n(Y,U;\IZ) & H_n(Y;\IZ). \ar[l]^\cong \\ } \]
The two left, horizontal arrows are isomorphisms by the long exact sequence for the homology of the pairs $(\widetilde{Y},\widetilde{U})$ and $(Y,U)$ and the homological conditions for $\widetilde{U}$ and $U$.  By assumption, the composition along the bottom row of the diagram is non-zero.  Consequently, $H_n(\widetilde{Y};\IZ) \neq 0$.  This completes the proof of the lemma.
\end{proof}

We conclude this section with a discussion on oriented sphere bundles.

\begin{defn}
An \emph{oriented $S^2$-bundle} over a manifold $X$ is a fibre bundle $\pi: \sE \to X$ with fibre $S^2$ such that there exists a collection of contractible open subsets $\{U_i\}_i$ of $X$, trivializations
\[ \phi_i: \pi^{-1}(U_i) \cong U_i \times S^2, \]
and generators
\[ \fo_i \in H^{2}(\sE|_{U_i};\IZ) \cong H^2(S^2;\IZ)  \]
so that over the intersections $U_{ij} = U_i \cap U_j$,
\[ \fo_i|_{\sE|_{U_{ij}}} = \fo_j|_{\sE|_{U_{ij}}}. \]
\end{defn}

Our oriented $S^2$-bundles will arise from the following situation:

\begin{lem}\label{lem:QuotientS2BundleOrientation}
Let $G$ be a Lie group.  Suppose that $X$ and $S^2$ admit smooth $G$-actions, where $G$ acts freely and properly on $X$ and $G$ acts via orientation preserving diffeomorphisms on $S^2$.  The projection
\[ \pi: X \times_G S^2 \to X/G\footnote{Here $X \times_G S^2$ denotes the quotient of $X \times S^2$ by the equivalence relation $(x,z) \sim (g \cdot x, g \cdot z)$ with $g \in G$.} \]
is an oriented $S^2$-bundle.
\end{lem}

\begin{proof}
By the Slice Theorem \cite[Proposition 2.2.2, Remark 2.2.3, and Section 4.1]{Palais_ExistenceOfSlices}, the projection
\[ \pi_X: X \to X/G\]
is a smooth principal $G$-bundle.  We have local trivializations over $U_i \subset X/G$
\[ \Phi_i: \pi_X^{-1}(U_i) \to U_i \times G, \gap \Phi_i(x) = (\pi_X(x),\phi_i(\pi_X(x))), \]
where $\phi_i: U_i \to G$.  The $\Phi_i$ are equivariant with respect to the $G$-action on $\pi_X^{-1}(U_i)$ and the $G$-action on $U_i \times G$ given by
\[ h \cdot (p,g) = (p,h \cdot g) \gap \mbox{for} \gap h \in G.\]
There are transition functions relating the $\Phi_i$,
\[ \Phi_{ij}: U_{ij} \times G \to U_{ij} \times G, \gap (p,g) \mapsto (p,g \cdot \phi_{ij}(p)), \]
where $\phi_{ij}: U_{ij} \to G$.  So $\Phi_i = \Phi_{ij} \circ \Phi_j$.  Notice that $\Phi_{ij}$ is $G$-equivariant.

These determine trivializations of $\pi$ over $U_i \subset X/G$,
\[ \Psi_i: \pi^{-1}(U_i) = (\pi_X^{-1}(U_i) \times_G S^2) \to U_i \times S^2, \gap \Psi_i = \sigma_i \circ (\Phi_i \times_G \Ione), \]
where
\[ \sigma_i: (U_i \times G) \times_G S^2 \to U_i \times S^2, \gap \sigma_i([(p,g),z]) = (p,g^{-1} \cdot z), \]
with transition functions relating the $\Psi_i$,
\[ \Psi_{ij}: U_{ij} \times S^2 \to U_{ij} \times S^2, \gap \Psi_{ij}(p,z) = (p, \phi_{ij}(p)^{-1} \cdot z).\]
In particular, $\Psi_i = \Psi_{ij} \circ \Psi_{j}$.  So we have shown that the projection is an $S^2$-bundle.

For orientability, fix an orientation $\fo \in H^2(S^2;\IZ)$.  Define $\fo_i \in H^2(\pi^{-1}(U_i);\IZ)$ via $\fo_i = \Psi_i^*(1 \otimes \fo)$.  Note, $\fo_i|_{\sE|_{U_{ij}}} = \fo_j|_{\sE|_{U_{ij}}}$ if and only if $\Psi_{ij}^* \fo = \fo$.  But the $G$-action on $S^2$ preserves $\fo$.  So by the above description of $\Psi_{ij}$, the desired equality follows.
\end{proof}

Finally, we recall a flavor of the Gysin sequence:

\begin{lem}\label{lem:SphereBundleProperties}
Let $\pi: \sE \to X$ be an oriented $S^2$-bundle.  There exists an exact sequence
\[ \xymatrix{ H_2(S^2;\IZ) \ar[r]^{i_x} & H_2(\sE;\IZ) \ar[r]^\pi & H_2(X;\IZ) \ar[r] & 0. } \]
\end{lem}

\begin{proof}
The claim follows from the homological Gysin sequence for an oriented sphere bundle
\[ \xymatrix{ \cdots \ar[r] & H_3(X) \ar[r] & H_0(X) \ar[r]^\tau & H_2(\sE) \ar[r]^\pi & H_2(X) \ar[r] & 0 }\]
and the fact that $\tau$ is identified with the map
\[ (i_x)_*: H_2(S^2;\IZ) \to H_2(\sE;\IZ) \] 
after identifying $[x] \in H_0(X)$ with $[\pi^{-1}(x)] \in H_2(S^2;\IZ)$ via the Thom isomorphism for the oriented $S^2$-bundle.
\end{proof}

\bibliographystyle{alpha}
\bibliography{references.bib}

\end{document}